\documentclass[a4paper,11pt]{amsart}

\usepackage[english]{babel}
\usepackage{amssymb}
\usepackage{a4wide}

\usepackage{amsmath}
\usepackage{enumerate}
\usepackage{colonequals} 
\usepackage{comment}
\usepackage{graphicx}
\usepackage{xcolor}
\usepackage[colorlinks=true, allcolors=gray]{hyperref}

\newtheorem{thm}{Theorem}
\newtheorem{lem}[thm]{Lemma}
\newtheorem{prop}[thm]{Proposition}
\newtheorem{cor}[thm]{Corollary}

\newtheorem*{que}{Question}
\theoremstyle{remark}
\theoremstyle{definition}
\newtheorem{defi}[thm]{Definition}

\newtheorem{ex}[thm]{Example}
\newtheorem{rem}[thm]{Remark}

\DeclareMathOperator{\id}{id}
\DeclareMathOperator{\Ann}{Ann}
\DeclareMathOperator{\N}{\mathbb{N}}
\DeclareMathOperator{\Z}{\mathbb{Z}}

\title{Ore extensions of abelian
groups with operators}
\date{\today}

\author{Per B\"{a}ck}
\address[Per B\"{a}ck]{Division of Mathematics and Physics, M\"{a}lardalen  University, SE-721 23  V\"{a}ster{\aa}s, Sweden}
\email[corresponding author]{per.back@mdu.se}

\author{Patrik Lundstr\"{o}m}
\address[Patrik Lundstr\"{o}m]{Department of Engineering Science, University West, SE-461 86  Trollh\"{a}ttan, Sweden}
\email{patrik.lundstrom@hv.se}

\author{Johan \"{O}inert}
\address[Johan \"{O}inert]{Department of Mathematics and Natural Sciences, Blekinge Institute of Technology, SE-371 79 Karlskrona, Sweden}
\email{johan.oinert@bth.se}
\address{\vspace{-15.5pt}\newline\indent{\upshape and}\vspace{3pt}\newline\indent Department of Engineering, University of Sk\"{o}vde, SE-541 28 Sk\"{o}vde, Sweden}

\author{Johan Richter}
\address[Johan Richter]{Department of Mathematics and Natural Sciences, Blekinge Institute of Technology, SE-371 79 Karlskrona, Sweden}
\email{johan.richter@bth.se}


\begin{document}

\begin{abstract}
Given a set $A$ and an abelian group $B$ with 
operators in $A$, in the sense of Krull and Noether, we introduce the Ore group extension
$B[x ; \sigma_B , \delta_B]$ as the 
additive group $B[x]$, with $A[x]$ as a set of operators. Here,
the action of $A[x]$ on $B[x]$ is defined 
by mimicking the 
multiplication used in the classical case where
$A$ and $B$ are the same ring.
We derive generalizations of Vandermonde's and
Leibniz's identities for this construction,
and they are then used to establish  
associativity criteria.
Additionally, we 
prove a version of Hilbert's basis theorem for this structure, under the assumption that the action of $A$ on 
$B$ is what we call weakly $s$-unital.
Finally, we apply these results to the case where 
$B$ is a left module over a ring $A$, and
specifically to the case where 
$A$ and $B$ coincide with a non-associative ring 
which is left distributive but not necessarily 
right distributive.
\end{abstract}

\maketitle

\section{Introduction}

In the landmark article \cite{Ore1933}, Ore 
introduced a variant of skew polynomial rings over
associative unital rings $R$, 
now known as \emph{Ore extensions}, 
denoted by $R[x ; \sigma , \delta]$. 
Since then, Ore extensions have become a fundamental 
construction in ring theory. 

Recall that the ring $R[x ; \sigma , \delta]$ is defined as the polynomial ring 
$R[x]$ as a left $R$-module, but with a modified 
associative multiplication, defined by the equality 
\begin{equation}\label{eq:commutation}
x r = \sigma(r)x + \delta(r)
\end{equation}
for $r \in R$. 
Here, $\sigma\colon R \to R$ is a ring endomorphism that 
preserves the identity, and $\delta\colon R \to R$ is a 
\emph{$\sigma$-derivation}, meaning that it is an additive map satisfying 
\begin{equation}
\delta(rs) = \sigma(r)\delta(s) + \delta(r)s
\end{equation}
for all $r,s \in R$. In the special case where 
$\sigma = \id_R$, the ring $R[x ; \id_R , \delta]$ is known as a 
\emph{differential polynomial ring}, 
and $\delta$ is called, simply, a \emph{derivation}.
If, additionally, $\delta = 0$, then $R[x ; \id_R , 0]$ is just the
standard polynomial ring $R[x]$.

Ore extensions appear
in the study of cyclic 
algebras, enveloping rings of solvable Lie algebras, 
and various types of graded rings such as group rings 
and crossed products; see e.g. 
\cite{Cohn1977,Jacobson1996,McConell2001,Rowen1988}. 
They also provide many important examples and counterexamples 
in ring theory \cite{Bergman1964,Cohn1961}. 
Moreover,
specific types of Ore extensions are used as
tools in different analytical contexts, such as in rings of
differential, pseudo-differential, and fractional differential
operators \cite{Goodearl1983}, and in $q$-Heisenberg 
algebras \cite{Hellstrom2000}.

Since Ore's original article \cite{Ore1933}, many 
generalizations of Ore extensions have been introduced.
These include modifications of the commutation rule 
\eqref{eq:commutation} and versions with multiple variables;
see e.g. 
\cite{Cojuhari2007,Cojuhari2012,Cojuhari2018,Malm1988,NOR2019,Smits1968}. Other authors have explored 
hom-associative, non-associative, non-unital, as well as ``flipped'' versions of Ore extensions; see e.g. \cite{AryapoorBack2025, BRS2018,LOR2023,NOR2018}.
Additional variants include the skew Laurent polynomial ring
$R[x^{\pm}; \sigma]$, the skew power series
ring $R[[x ; \sigma]]$, and the skew Laurent series 
ring $R((x; \sigma))$, as discussed in \cite{Goodearl2004}.

Various properties of these constructions, such as
conditions under which they are integral domains, principal
ideal domains, prime, or simple, have been
extensively studied by numerous authors; see e.g. the surveys
\cite{Cozzens1975,McConell2001} and the references therein.
One of the most fundamental results in this direction,
with many applications in ring theory and algebraic 
geometry, is the following:

\begin{thm}[Hilbert's basis theorem]\label{thm:Hilbertsbasistheorem}
Suppose that $R$ is an associative unital ring. 
Then $R[x]$ is left Noetherian if and only if 
$R$ is left Noetherian.
\end{thm}

The ``if'' direction of Hilbert's basis theorem
has been shown to hold for many Ore extensions as well:

\begin{thm}\label{thm:HilbertsbasistheoremOre}
Suppose that $R$ is an associative, unital, left  
Noetherian ring with a ring automorphism $\sigma$ and a
$\sigma$-derivation $\delta$. Then $R[x;\sigma,\delta]$
is also left Noetherian. 
\end{thm}

Similar results hold for the rings
$R[x^{\pm}; \sigma]$, $R[[x ; \sigma]]$, and 
$R((x; \sigma))$. For detailed proofs of these results,
as well as of Theorem~\ref{thm:HilbertsbasistheoremOre},
see e.g. \cite{Goodearl2004}.

Another type of generalization of 
Hilbert's basis theorem
involves \emph{polynomial modules}
(see Theorem~\ref{thm:varadarajan}).
More specifically, consider the left 
$R[x]$-module $M[x]$, where 
$M$ is a left $R$-module, and the module structure is 
defined by the biadditive extension of the relations
$(r x^m)(a x^n) = (ra) x^{m+n}$ for 
$r \in R$, $a \in M$, and 
non-negative integers $m$ and $n$.
For a submodule $N$ of $M$, we define
$R^{-1}N \colonequals \{ x \in M \mid Rx \subseteq N \}$.
The module $M$ is said to have 
\emph{property (F)} if, for any submodule $N$ of $M$, 
we have $R^{-1} N = N$.
It can easily be shown that an $R$-module $M$
has property (F) if and only if it is  
\emph{$s$-unital}, meaning that for 
each $a \in M$
there exists $r \in R$ with $ra = a$.
In  \cite[Thm.~A]{Varadarajan1982}
Varadarajan proves the following:

\begin{thm}[Varadarajan \cite{Varadarajan1982}]\label{thm:varadarajan}
Suppose that $R$ is an associative, but not necessarily unital,
ring and $M$ is 
a left $R$-module. Then the left $R[x]$-module $M[x]$ is
Noetherian if and only if 
$M$ is Noetherian and has property (F).
\end{thm}

After reflecting upon the constructions and results outlined
above, the authors of the present article were 
prompted to explore the following:

\begin{que}\label{que:question}
Is it possible to define a class of
``Ore module extensions'' so that these simultaneously 
generalize polynomial modules and classical Ore extensions? 
If so, can algebraical
structure results for Ore module extensions, such as
associativity and a Hilbert's basis theorem,
be established?
\end{que}

In this article, we provide affirmative answers
to these questions. We show that not only is it possible to
quite naturally define \emph{Ore module extensions}
$M[x ; \sigma_M, \delta_M]$
over \emph{Ore ring extensions} 
$R[x ; \sigma_R, \delta_R]$,
but one can also define, more generally, 
\emph{Ore group extensions}
$B[x ; \sigma_B, \delta_B]$
for any \emph{abelian group $B$ with operators} in a set $A$; the latter concept first introduced by Krull \cite{Krull1925} for finite abelian groups as \emph{generalized finite abelian groups} and then extensively studied by Noether (see e.g. \cite{Noether1929}) in the more general setting under its current name. Importantly, 
this broader approach of ours allows for 
applications in both \emph{non-associative} and 
\emph{non-unital} contexts, 
as well as in situations where a module or ring
multiplication is \emph{left distributive} but not 
necessarily right distributive.\newpage

Here is an outline of this article.

In Section \ref{sec:groupswithoperators},
we begin by reviewing some concepts and results 
related to the actions of sets on groups. 
In particular, we establish a few folkloristic results on
Noetherian groups with operators (see
Propositions~\ref{thm:threeequivalent}-\ref{prop:directproduct}).
We also introduce two seemingly new concepts: stable homomorphisms for such groups
(see Definition~\ref{def:stable}) and 
weakly $s$-unital actions (see Definition~\ref{def:sunital}).

In Section \ref{sec:Ore}, 
we introduce Ore group extensions
of abelian groups with sets of operators 
(see Definition~\ref{def:oreextension}).
Using general versions of the 
Vandermonde and Leibniz identities 
(see Proposition~\ref{prop:Vandermonde} and 
Proposition~\ref{prop:Leibniz}), we derive criteria for 
associativity of such structures
(see Theorem~\ref{thm:MimpliesMx}).

In Section \ref{sec:hilbertsbasistheorem},
we establish a version of Hilbert's basis
theorem (see Theorems~\ref{thm:hilbertAB} and 
\ref{thm:polynomialnoether})
for Ore extensions of abelian groups.
In particular, we extend Theorem~\ref{thm:varadarajan}
by generalizing it to the context of 
groups with operators
(see Theorem~\ref{thm:polynomialnoether}).

In Section \ref{sec:oremodules}, we
introduce Ore module extensions
(see Definition~\ref{def:oremodule})
and apply results from earlier sections to 
these new structures. In particular, we 
obtain
a 
generalization of Theorem~\ref{thm:Hilbertsbasistheorem}
for rings $R$ which are left distributive and 
weakly $s$-unital, but not necessarily associative
or right distributive (see 
Corollary~\ref{cor:leftdistributive}). 
This allows us to address many previously 
unreachable cases for $R$, 
such as all unital and Noetherian rings which are 
alternative or Jordan,
all Dickson’s left near-fields \cite{Pilz1983}, 
and all algebras generated by either of 
the Cayley--Dickson or 
Conway--Smith doubling procedures \cite{Lundstrom2012}.

\section{Groups with operators}\label{sec:groupswithoperators}

Throughout this article, we put
$\Z \colonequals 
\{ 0, \pm 1, \pm 2, \ldots \}$,
$\N \colonequals \{ n \in \Z \mid
n \geq 0 \}$ and
$\N_+ \colonequals \{ n \in \N \mid n > 0 \}$.
We recall the following notions and results
from \cite{Bourbaki1989}
(see also \cite{Isaacs2009,Jacobson1951}).
Let $A$ and $B$ be sets. 
Suppose that $\alpha$
is an \emph{action} of $A$ on $B$.
By this, we mean that $\alpha$ is a 
function
from $A$ to the set $B^B$ of functions 
from $B$ into itself. 
By abuse of notation,
we will often suppress
$\alpha$ and write $ab \colonequals
\alpha(a)(b)$, for $a \in A$ and
$b \in B$, placing $a$ to the \emph{left}
of $b$, and say that \emph{$B$ is a set
with operators in $A$}.

Let $S$ be a subset of $B$.
Then $S$ is called \emph{stable} if
$as \in S$, for $a \in A$ and $s \in S$.
The intersection of the family 
of stable subsets of $B$
containing $S$ is called the 
stable subset of $B$ 
\emph{generated} by $S$;
this set is denoted $\overline{S}$.
We put $A^0 S \colonequals S$ 
and for $n \in \N_+$
we let $A^n S$ denote the set of all 
elements of the form
$a_1(a_2( \cdots (a_{n-1}(a_n s)) \cdots ))$
for $a_1,\ldots,a_n \in A$
and $s \in S$. From 
\cite[Ch.~I, \S 3.2, Discussion after Def.~2]{Bourbaki1989} 
we extract the following:

\begin{prop}\label{prop:Sclosure}
Let $S$ be a subset of $B$.
Then $\overline{S} = 
\cup_{n \in \N} (A^n S)$.
\end{prop}

Suppose henceforth that $(B,\cdot)$ 
is a \emph{group with operators in $A$.} 
By this, we mean that the elements of $A$
act as group endomorphisms of $B$, that is
$a(b \cdot c) = 
(a b) \cdot (a c)$, for $a \in A$ and
$b,c \in B$.
A subgroup of $B$ is called \emph{$A$-stable} (or simply \emph{stable})
if it is stable as a subset of $B$.
Let $e$ denote the identity element of $B$.
Since $ae = e$, for $a \in A$, it follows that
$\{ e \}$ is always a 
stable subgroup of $B$.
Furthermore, trivially, $B$ is a stable subgroup of itself. Let $S$ be a 
nonempty subset of $B$.
The intersection of the family of stable
subgroups of $B$ that contain $S$
is called the stable subgroup of
$B$ \emph{generated} by $S$
and is denoted by 
$\langle S \rangle$. If $C$ is a 
stable subgroup of $B$ such that 
$C = \langle T \rangle$ for some
finite subset $T$ of $B$, then
$C$ is said to be \emph{finitely generated}
by $T$. From \cite[Ch.~I, \S 4.3, Prop.~2]{Bourbaki1989} we recall the following:

\begin{prop}\label{prop:generated}
Let $S$ be a nonempty subset of $B$.
Then $\langle S \rangle$
equals the set of all products of the form
$b_1 \cdots b_n$, for $n \in \N_+$,
where for each $k \in \{ 1,\ldots,n \}$,
$b_k \in \overline{S}$ or 
$b_k^{-1} \in \overline{S}$.
\end{prop}

Consider the partially ordered set 
of stable subgroups of $B$, ordered
by inclusion. We say that
$B$ is \emph{stably Noetherian} if this partially
ordered set satisfies 
the ascending chain condition.
The next result is folklore.
Due to the lack of an appropriate reference,
we include a proof (see  
\cite[Thm.~11.4]{Isaacs2009} 
for the abelian groups case).

\begin{prop}\label{thm:threeequivalent} 
The following assertions are equivalent:

\begin{enumerate}[{\rm (i)}]
\item $B$ is stably Noetherian;

\item any nonempty family of
stable subgroups of $B$ has a maximal
element;

\item every stable subgroup of 
$B$ is finitely generated.
\end{enumerate}
\end{prop}

\begin{proof}
(i)$\Rightarrow$(ii):
Suppose that (i) holds.
Seeking a contradiction, 
suppose that $\mathcal{F}$ is a nonempty 
family of stable subgroups of $B$
with no maximal element. Take 
$B_1 \in \mathcal{F}$. Since $B_1$ is
not maximal, there is $B_2 \in 
\mathcal{F}$ with $B_1 \subsetneq B_2$.
Continuing in this way we can construct
an infinite ascending 
chain of stable subgroups
of $B$ that does not stabilize, which is 
a contradiction.

(ii)$\Rightarrow$(iii):
Suppose that (ii) holds. Let $C$ be a
stable subgroup of $B$.
Define $\mathcal{F}$ to be the set 
of all finitely generated 
stable subgroups of $C$.
Note that 
$\mathcal{F} \neq \emptyset$, because $\{ e \} \in \mathcal{F}$.
By hypothesis, $\mathcal{F}$ has a 
maximal element $M$. By the definition
of $\mathcal{F}$ we have $M \subseteq C$.
Seeking a contradiction, suppose that
$M \subsetneq C$. Take $c \in C \setminus M$,
$k \geq 1$ and $m_1,\ldots,m_k \in M$
such that $M = \langle 
m_1,\ldots,m_k \rangle$.
Put $N \colonequals \langle 
m_1,\ldots , m_k , c \rangle$.
Then $N\in\mathcal{F}$ with $M\subsetneq N$, which contradicts the maximality of $M$. Thus $M=C$, whence $C$ is finitely generated.

(iii)$\Rightarrow$(i):
Suppose that (iii) holds. Let 
$C_1 \subseteq C_2 \subseteq \cdots$ be
an ascending chain of 
stable subgroups of $B$. Then, clearly, 
$C \colonequals \cup_{i \geq 1} C_i$
is a stable subgroup of $B$.
By hypothesis, there is $n \in \N_+$,
$i_1 \leq \dots \leq i_n$
in $\N_+$ and $c_{i_k} \in C_{i_k}$,
for $k \in \{ 1,\ldots,n \}$, with
$C = \langle c_{i_1}, \ldots ,
c_{i_n} \rangle$.
Let $k \geq i_n$.
Then $C_k \supseteq C_{i_n}$.
Since $c_{i_1},\ldots,c_{i_n} \in C_{i_n}$,
it follows that $C = \langle c_{i_1}, 
\ldots , c_{i_n} \rangle \subseteq C_{i_n}$.
Thus, $C_k \subseteq C \subseteq C_{i_n}$,
so that $C_k = C_{i_n}$.
In particular, $C_{i_n} = C_{i_n+1} = \cdots$, 
showing that the chain stabilizes.
\end{proof}

Let $C$ be a normal stable subgroup 
of $B$. The quotient group $B/C$ is
then a group with operators in $A$ if we 
put $a(b C) \colonequals (ab) C$
for $a \in A$ and $b \in B$
(see \cite[Ch.~I, \S 4.4, Thm.~2]{Bourbaki1989}).
The next result is also folklore.
Due to the lack of an appropriate reference,
we include a proof.

\begin{prop}\label{prop:subquotient}
Let $C$ be a normal stable 
subgroup of $B$.
Then $B$ is stably Noetherian if and only 
if $C$ and $B/C$ are stably Noetherian.
\end{prop}

\begin{proof}
We generalize the proof for the abelian case
given in \cite[Thm.~11.6]{Isaacs2009}. 

Suppose that $B$ is stably Noetherian.
Any stable subgroup $D$ 
of $C$ is also a stable subgroup of $B$.
Hence, by assumption and
Proposition~\ref{thm:threeequivalent},
$D$ is finitely generated. Thus, 
$C$ is stably Noetherian. Now we show that
$B/C$ is stably Noetherian. 
By \cite[Ch.~I, \S 4.6, Cor.~1]{Bourbaki1989},
the quotient map $B \to B/C$ induces 
an inclusion preserving bijection between
the family of stable subgroups of $B$
containing $C$ and the family of stable
subgroups of $B/C$. From this bijection
it follows that $B/C$ is stably Noetherian,
since $B$ is stably Noetherian.

Suppose that $C$ and $B/C$
are stably Noetherian. Let 
$D_1 \subseteq D_2 \subseteq \cdots$
be an ascending chain of stable subgroups
of $B$. Then $D_1 \cap C \subseteq 
D_2 \cap C \subseteq \cdots$ and
$D_1 C \subseteq D_2 C  \subseteq \cdots$ are ascending chains of stable
subgroups of $C$ and $B/C$ respectively.
By assumption 
there is $n \in \N_+$ 
such that for any $i \geq n$, the
equalities
$D_i \cap C = D_n \cap C$ and
$D_i C = D_n C$ hold. Take $i \geq n$.
We wish to show that $D_i = D_n$.
By assumption $D_n \subseteq D_i$.
Now we show that $D_i \subseteq D_n$.
Take $d_i \in D_i$.
Then $d_i \in D_i C = D_n C$ so that
$d_i = e_n c$ for some $e_n \in D_n$
and $c \in C$. Then
$e_n^{-1} \in D_n \subseteq D_i$ 
so that $c = e_n^{-1} d_i 
\in C \cap D_i = C \cap D_n \subseteq D_n$.
Hence, $d_i = e_n c \in D_n D_n \subseteq D_n$.
\end{proof}

Let $\{ B_i \}_{i \in I}$ be a 
family of groups with operators in $A$, and let $B$ be the direct product 
$\prod_{i\in I} B_i$. Consider the action of $A$ on
$B$ defined by $a (b_i)_{i \in I}
\colonequals (a b_i)_{i \in I}$,
for $a \in A$ and $b_i \in B_i$.
With this structure, $B$ is a group
with operators in $A$.

\begin{prop}\label{prop:directproduct}
Let $n \in \N_+$, and let
$\{B_i\}_{i=1}^n$ be a family of groups 
with operators in $A$.
Consider $B \colonequals B_1 \times \cdots \times B_n$ 
as a group with operators in $A$,
in the sense defined above.
Then $B$ is stably Noetherian if and only if $B_i$ is 
stably Noetherian for every $i \in \{1,\ldots,n\}$.
\end{prop}

\begin{proof}
The ``only if'' statement follows from
Proposition~\ref{prop:subquotient} since 
each $B_i$ is isomorphic to a normal
stable subgroup of $B$.
Now we show the ``if'' statement.
Suppose that each $B_i$ is stably Noetherian.
We prove, by induction on $n$,
that $ B_1 \times
\cdots \times B_n$ is stably Noetherian.
The claim trivially holds if $n=1$.
Suppose now that 
$B_1 \times \cdots \times B_k$ is
stably Noetherian for some $k \in \N_+$.
Put $B \colonequals B_1 \times
\cdots \times B_{k+1}$ and  
$C \colonequals B_{k+1}$. Then $C$ is
isomorphic to
a stably Noetherian normal stable subgroup 
of $B$ such that $B/C \cong 
B_1 \times \cdots \times B_k$.
By the induction hypothesis, $B/C$
is therefore also stably Noetherian. By Proposition~\ref{prop:subquotient}, 
$B$ is stably Noetherian.
\end{proof}

\begin{defi}[$A$-stable homomorphism]\label{def:stable}
Let $B$ and $C$ be groups with 
operators in $A$.
Suppose that $f\colon B \to C$ is a 
group homomorphism. Then
we say that $f$ is \emph{$A$-stable} 
if for every $b \in B$, the inclusion 
$f(Ab) \subseteq A f(b)$ holds.
\end{defi}

\begin{prop}\label{prop:kerf}
Let $B$, $C$, and $D$ be groups with 
operators in $A$.
Suppose that $f\colon B \to C$ and 
$g\colon C \to D$ are 
$A$-stable group homomorphisms.
Then $g \circ f\colon B \to D$ is $A$-stable
and $\ker(f)$ is an $A$-stable subgroup 
of $B$.
\end{prop}

\begin{proof}
Take $b \in B$. Then
$g(f(Ab)) \subseteq
g(A f(b)) \subseteq A g(f(b))$. 
Therefore, $g \circ f$ is $A$-stable.
Also $f(A \ker(f)) \subseteq A f(\ker(f)) = Ae = \{e\}$.
Thus, $\ker(f)$ is $A$-stable. 
\end{proof}

\begin{prop}\label{prop:surinj}
Let $B$ be a stably Noetherian group with 
operators in $A$. Suppose that
$f\colon B \to B$ is a surjective 
$A$-stable group homomorphism.
Then $f$ is bijective.
\end{prop}

\begin{proof}
We want to show that 
$\ker(f) = \{ e \}$, i.e. $f$ is injective. 
By Proposition~\ref{prop:kerf}, all the 
function compositions $f^k \colonequals 
f \circ \cdots \circ f$ ($k$
functions), for $k \in \N_+$, are
$A$-stable group homomorphisms. By 
Proposition~\ref{prop:kerf} again,
$\ker(f^k)$ 
is an $A$-stable subgroup of $B$
for every $k \in \N_+$.
Clearly, $\ker(f) \subseteq \ker(f^2)
 \subseteq \cdots$.
Since $B$ is stably Noetherian, there is
$n \in \N_+$ such that
$\ker(f^n) = \ker(f^{n+1}) = \cdots$.
We claim that 
$\ker(f^n) \cap {\rm im}(f^n) = \{ e\}$.
Let us assume for a moment that this
claim holds. Since $f$ is 
surjective, it follows that
$\{ e \} = \ker(f^n) \cap {\rm im}(f^n) =
\ker(f^n) \cap B = \ker(f^n)$.
Thus, $\ker(f) \subseteq \ker(f^n) = 
\{ e \}$ so that $\ker(f) = \{ e \}$.
Now we show the claim. Trivially,
$\ker(f^n) \cap {\rm im}(f^n)
\supseteq \{ e \}$. Now we show the 
reversed inclusion. To this end,
take $x \in \ker(f^n) \cap {\rm im}(f^n)$.
Then $f^n(x) = e$ and there is $y \in B$
with $f^n(y) = x$. Thus, $f^{2n}(y) =
f^n(f^n(y)) = f^n(x) = e$. Hence,
$y \in \ker(f^{2n}) = \ker(f^n)$ so that
$x = f^n(y) = e$. Therefore,
$\ker(f^n) \cap {\rm im}(f^n) \subseteq
\{ e\}$.
\end{proof}

\begin{defi}[$\tau$-twist]
Let $B$ and $C$ be groups with operators 
in $A$. Suppose that $f\colon B \to C$ is a group
homomorphism. Let $\tau$ be a map $A \to A$. 
We say that $f$ is 
\emph{$\tau$-twisted} if for all $a \in A$
and all $b \in B$, the equality 
$f(ab) = \tau(a) f(b)$ holds.
\end{defi}

\begin{prop}\label{prop:tautwisted}
Let $B$ and $C$ be groups with operators 
in $A$. Suppose that $f\colon B \to C$ is a 
$\tau$-twisted group
homomorphism for some map $\tau\colon A \to A$.
Then $f$ is $A$-stable.
\end{prop}

\begin{proof}
Take $b \in B$. Then $f(Ab) = \tau(A)f(b)
\subseteq A f(b)$.
\end{proof}

\begin{defi}[Weakly $s$-unital action]\label{def:sunital}
Suppose that $B$ is a group with operators
in $A$. Let $S$ be a nonempty subset 
of $B$. Put $\widetilde{S}
\colonequals \cup_{n \in \N_+} (A^n S)$. 
Let $[S]$
denote the set of all $b_1 \cdots b_n$,
for $n \in \N_+$, where for each 
$k \in \{ 1,\ldots,n \}$, $b_k \in 
\widetilde{S}$ or $b_k^{-1} \in \widetilde{S}$. 
Clearly, $[S]$ is a stable subgroup 
of $B$. We say that the action of $A$
on $B$ is \emph{$s$-unital} (resp. \emph{weakly $s$-unital}) if 
for every $b \in B$ the relation
$b \in Ab$ (resp. $b \in [b]$) holds.
\end{defi}

\begin{prop}\label{prop:weak-s-unital-action}
The following assertions
are equivalent:
\begin{enumerate}[{\rm (i)}]
\item the action of $A$ on $B$
is weakly $s$-unital;

\item for every nonempty subset
$S$ of $B$, the equality 
$\langle S \rangle = [S]$ holds.
\end{enumerate}
\end{prop}

\begin{proof}
(i)$\Rightarrow$(ii):
Suppose that (i) holds. Let $S$ be a 
nonempty subset of $B$.
By Proposition~\ref{prop:Sclosure}, $\widetilde{S} \subseteq 
\overline{S}$ and hence, 
by Proposition~\ref{prop:generated}, 
$[S] \subseteq  \langle S \rangle$.
Now we show the reversed inclusion.
Take $s \in S$. By (i),  $s \in [s] \subseteq [S]$.
Thus, $[S]$ is a stable subgroup of $B$
containing $S$. By the definition of 
$\langle S \rangle$ we get that 
$\langle S \rangle \subseteq [S]$.
This proves (ii).

(ii)$\Rightarrow$(i):
Suppose that (ii) holds. If $b \in B$,
then $b \in \langle b \rangle = [b]$.
This proves (i).
\end{proof}

\begin{ex}\label{ex:inverseaction}
Suppose that $B$ is an abelian group and  
$A = \{ a \}$ is a singleton set. Define the
action of $A$ on $B$ by $ab \colonequals b^{-1}$ for $b \in B$.
Since $a(ab) = (b^{-1})^{-1} = b$, the action of 
$A$ on $B$ is weakly $s$-unital. Note, however, that
since $a b  \neq b$ for all $b \in B$ with 
$b^2 \neq e$, the action is $s$-unital if 
and only if $B$ is a Boolean group.
\end{ex}

\section{Ore group extensions}\label{sec:Ore}

Throughout this section,
$B$ is an \emph{abelian} group 
with operators in a nonempty set $A$,
$+$ denotes the group operation in $B$,
and $0$ denotes the identity 
element of~$B$.

\begin{defi}[Zero element]\label{defi:zeroelement}
We always assume that $A$ has 
a \emph{zero element}.
By this, we mean an element
$\epsilon \in A$ such that for any
$b \in B$, $\epsilon b = 0$. We will assume that $\epsilon$ 
is fixed.
By abuse of notation, we
put $0 \colonequals \epsilon$,
so that the convenient 
equality $a 0 = 0 b = 0$ holds
for all $a \in A$ and $b \in B$.
\end{defi}

\begin{defi}[Polynomial group]
By a \emph{polynomial} over $A$ 
we mean a formal sum
$\sum_{i \in \N} a_i x^i$, 
where $a_i \in A$,
for $i \in \mathbb{N}$,
and $a_i = 0$
for all but finitely many $i \in 
\mathbb{N}$. The set of 
polynomials over $A$ is denoted by $A[x]$.
We define $B[x]$ similarly and equip it with an abelian
group structure in the following way.
If $\sum_{i \in \N} b_i x^i,
\sum_{j \in \N} b'_j x^j \in B[x]$,
then we put
\[ \sum_{i \in \N} b_i x^i
+ 
\sum_{j \in \N} b_j' x^j 
\colonequals
\sum_{i \in \N}(b_i + b_i') x^i.\]
The zero polynomial is defined to be 
$0 \colonequals \sum_{i \in \N} 0 x^i$.
\end{defi}

\begin{defi}[$\pi$-maps]
From now on, 
let $\sigma_B$ and $\delta_B$
be group endomorphisms of $B$.
Take $i,j \in \mathbb{N}$.
We define $\pi_j^i \colon B \to B$ 
in the following way.
If $i \geq j$, then
we let
$\pi_j^i \colon B \to B$ denote
the sum of all $\binom{i}{j}$ compositions of $j$ instances of $\sigma_B$ and $i-j$ instances of $\delta_B$.
If $i < j$, then we put
$\pi_j^i \colonequals 0$. 
\end{defi}

\begin{defi}[Ore group extension]\label{def:oreextension}
The \emph{Ore group extension}
$B[x;\sigma_B,\delta_B]$ 
is the abelian group $B[x]$ 
with $A[x]$ as a set of operators,
the action being given by 
\[
\Biggl( \sum_{i \in \N} a_i x^i \Biggr)
\Biggl( \sum_{j \in \N} b_j x^j \Biggr)
\colonequals
\sum_{i,j,k \in \N} 
\left( a_i \pi^i_k(b_j) \right)
x^{k+j}
\]
for $\sum_{i \in \N} a_i x^i \in A[x]$ and 
$\sum_{j \in \N} b_j x^j \in B[x]$.
By abuse of notation, we write $B[x]$ for $
B[x;\id_B,0_B]$.
\end{defi}

\begin{prop}[Vandermonde's identity]\label{prop:Vandermonde}
$\sum_{i\in\mathbb{N}} 
\pi_i^k \circ \pi_{j-i}^n  =
\pi_j^{k+n}$ for $j,k,n \in \N$.
\end{prop}

\begin{proof}
Take $j,k,n \in \mathbb{N}$. 
Let $X$ denote the set of words of length
$k+n$ having as letters exactly $j$ 
copies of $\sigma_B$ and $k+n-j$
copies of $\delta_B$.
The sum $\pi_j^{k+n}$
equals the sum of all words from $X$,
where we interpret each word as the 
function corresponding to the 
composition of its letters.
For each $i$, let $X_i$ denote the 
set of words of length $k+n$ having exactly
$i$ copies of $\sigma_B$ amongst its first
$k$ letters, and exactly $j-i$
copies of $\sigma_B$ amongst its remaining
$n$ letters.
Fix $i$. Expanding the sums $\pi_i^k$ and
$\pi_{j-i}^n$ it follows that 
the terms in the resulting 
sum $\pi_i^k \circ \pi_{j-i}^n$ 
are in bijective correspondence with the
words in $X_i$.
Since, clearly, $\{ X_i \}_{i \in \mathbb{N}}$
is a partition of $X$, the sought identity
follows.
\end{proof}

\begin{cor}\label{cor:oneshift}
$\pi_{j-1}^k \circ \sigma_B 
+ \pi_j^k \circ \delta_B = \pi_j^{k+1}$
for  $j,k \in \N$. 
\end{cor}

\begin{proof}
This follows from Proposition~\ref{prop:Vandermonde}
with $n = 1$.
\end{proof}

\begin{defi}[$\sigma$-derivation and $\sigma$-twist]
Suppose 
from now on
that $\sigma_A$ and $\delta_A$
are maps $A \to A$,
and that $\sigma_B$ and $\delta_B$
are additive maps $B \to B$.
We say that $\delta_B$ is a 
\emph{$\delta_A$-twisted $\sigma_A$-derivation} 
if $\delta_B(ab)= \sigma_A(a)\delta_B(b)
+\delta_A(a)b$
for all $a \in A$
and $b \in B$.
Additionally, we say that
$\sigma_B$ is \emph{$\sigma_A$-twisted} 
if $\sigma_B(ab) = \sigma_A(a) \sigma_B(b)$
for all $a \in A$ and $b \in B$.
\end{defi}

\begin{prop}[Leibniz's identity]\label{prop:Leibniz}
Suppose that $\delta_B$ is a $\delta_A$-twisted 
$\sigma_A$-derivation and 
$\sigma_B$ is $\sigma_A$-twisted. 
Take $a \in A$, $b \in B$ and $i,m \in \N$.
Then $\pi_i^m(ab) = 
\sum_{k\in\mathbb{N}} \pi_k^m(a) \pi_i^k(b).$
\end{prop}

\begin{proof}
We prove the sought identity by
induction on $m$. It clearly holds for $m=0$.
Suppose that it holds
for some $m \in \N$.
Using Corollary~\ref{cor:oneshift}, we get that
\[
\begin{array}{rcl}
 \pi_i^{m+1}(ab) 
& = & 
\pi_{i-1}^m (\sigma_B(ab) ) + 
\pi_i^m( \delta_B(ab) ) \\
& = & 
\pi_{i-1}^m (\sigma_A(a)\sigma_B(b) ) + 
\pi_i^m( \sigma_A(a)\delta_B(b) ) +
\pi_i^m( \delta_A(a) b ) \\
& = & 
\sum_{k\in\mathbb{N}} \pi_k^m(\sigma_A(a)) 
\pi_{i-1}^k(\sigma_B(b)) + 
\sum_{k\in\mathbb{N}} \pi_k^m(\sigma_A(a)) 
\pi_i^k(\delta_B(b)) \\
&  & + \
\sum_{k\in\mathbb{N}} \pi_k^m(\delta_A(a)) \pi_i^k(b) \\
& = & 
\pi_{i-1}^m(\sigma_A(a)) \sigma_B^i(b) 
+ 
\sum_{k=i}^m \left[ 
\pi_k^m(\sigma_A(a)) 
\pi_{i-1}^k(\sigma_B(b)) +
\pi_k^m(\sigma_A(a)) 
\pi_i^k(\delta_B(b)) 
\right] \\
&  & 
+ \
\sum_{k\in\mathbb{N}} \pi_k^m(\delta_A(a)) \pi_i^k(b) \\
&=& 
\pi_{i-1}^m(\sigma_A(a)) \sigma_B^i(b) +
\sum_{k=i}^m \pi_k^m(\sigma_A(a)) 
\pi_i^{k+1}(b)
+ \sum_{k\in\mathbb{N}} 
\pi_k^m(\delta_A(a)) \pi_i^k(b) \\
&=& 
\sum_{k=i}^{m+1} 
\pi_{k-1}^m(\sigma_A(a)) \pi_i^k(b) + 
\sum_{k=i}^{m+1} \pi_k^m(\delta_A(a)) \pi_i^k(b) 
=\sum_{k=i}^{m+1} \pi_k^{m+1}(a) \pi_i^k(b).
\end{array}
\]
\end{proof}

\begin{prop}\label{prop:mixed}
Suppose that $\delta_B$ is a $\delta_A$-twisted
$\sigma_A$-derivation and $\sigma_B$ is
$\sigma_A$-twisted. Take $a \in A$,
$b \in B$ and $j,m,n\in\mathbb{N}$.
Then $\sum_{i \in \N}
\pi_i^m(a \pi_{j-i}^n(b)) =
\sum_{i \in \N}
\pi_i^m(a) \pi_j^{i+n}(b).$
\end{prop}

\begin{proof}
By Leibniz's and Vandermonde's 
identities, we get that
\[
\begin{array}{rcl}
\sum_{i\in\mathbb{N}}
\pi_i^m(a \pi_{j-i}^n(b))  
 &=& \sum_{i\in\mathbb{N}} \sum_{k \in\N}
\pi_k^m(a) 
(\pi_i^k \circ \pi_{j-i}^n)(b) =  
\sum_{k\in\mathbb{N}} \sum_{i \in \N} \pi_k^m(a)
(\pi_i^k \circ \pi_{j-i}^n)(b) \\
& = & 
\sum_{k\in\mathbb{N}} \pi_k^m(a) \sum_{i\in\mathbb{N}} 
(\pi_i^k \circ \pi_{j-i}^n)(b) 
= \sum_{k \in \N } 
\pi_k^m(a) \pi_j^{k+n}(b). 
\end{array} 
\]
\end{proof}

\begin{defi}[Associativity]\label{def:assoc}
From now on $C$ is an abelian 
group with operators in $A$ 
and in $B$. We say that 
$(A,B,C)$ is \emph{associative} 
if $(ab)c = a(bc)$
for all $a \in A$, 
$b \in B$ and $c \in C$.
\end{defi}

Henceforth, let $\sigma_C$ and $\delta_C$ be additive maps $C \to C$.
Note that the Ore group extension $C[x;\sigma_C, \delta_C]$ is an abelian group with operators in $A[x]$ and in $B[x]$, under our stated conventions on $A$, $B$, and $C$. Also note that $C[x;\sigma_C, \delta_C]$ can 
be seen as an abelian group with operators in
either of the sets $B[x;\sigma_B,\delta_B]$ 
or $B[x]$.

\begin{thm}\label{thm:MimpliesMx}
Let $(A,B,C)$ be associative. 
Suppose that $\delta_C$ is a 
$\delta_B$-twisted
$\sigma_B$-derivation and $\sigma_C$
is $\sigma_B$-twisted. Then 
$(A[x],B[x;\sigma_B,\delta_B],
C[x;\sigma_C,\delta_C])$ is 
associative.
\end{thm}

\begin{proof}
Let $\alpha \colonequals
\sum_{i \in \N} a_i x^i \in A[x]$,
$\beta \colonequals
\sum_{j \in \N} b_j x^j \in B[x]$,
and 
$\gamma \colonequals
\sum_{k \in \N} c_k x^k \in C[x]$.
By Proposition~\ref{prop:mixed}, and 
the associativity of $(A,B,C)$,
we get that
\[
\begin{array}{rcl}
(\alpha \beta)\gamma 
& = &
\sum_{i,j,p \in \N} 
a_i \pi_p^i(b_j) x^{p+j}
\sum_{k \in \N} c_k x^k 
= \sum_{i,j,k,p,q \in \N}
\left( 
a_i \pi_p^i(b_j)
\right) \pi_q^{p+j}(c_k) x^{q+k} \\
& = &
\sum_{i,j,k,p,q \in \N}
 a_i \left( \pi_p^i(b_j)
\pi_q^{p+j}(c_k) \right) x^{q+k} 
= \sum_{i,j,k,q \in \N}
 a_i \left(\sum_{p \in \N} \pi_p^i(b_j)
\pi_q^{p+j}(c_k) \right)  x^{q+k} \\
& = & 
\sum_{i,j,k,q \in \N}
 a_i \left(\sum_{p \in \N} \pi_p^i(b_j
\pi_{q-p}^{j}(c_k))\right)  x^{q+k} 
\quad 
\mbox{[Put $r \colonequals q-p$]} \\
& = & 
\sum_{i,j,k,r \in \N}
 a_i \sum_{p \in \N} \pi_p^i(b_j
\pi_r^j(c_k))  x^{p+r+k} 
=\sum_{i,j,k,r,p \in \N}
a_i \left( \pi_p^i(b_j
\pi_r^j(c_k)) \right)  x^{p+r+k} \\
& = &  
\alpha \sum_{j,k,r \in \N}
b_j \pi_r^j(c_k) x^{r+k} 
 =  \alpha (\beta \gamma).
\end{array}
\]
\end{proof}

In analogy with the situation for modules, 
we define the \emph{annihilator of $A$ in $C$}
to be the set 
$\Ann_C(A) \colonequals \{ c \in C \mid 
Ac = \{ 0 \} \}$.

\begin{cor}\label{cor:assocEquiv}
Let $\sigma_C$ and 
$\delta_C$ be group endomorphisms of $C$. 
Suppose that $\Ann_C(A) = \{ 0 \}$.
Then the following assertions are equivalent:

\begin{enumerate}[{\rm (i)}]
\item $(A[x],B[x;\sigma_B,\delta_B],
C[x;\sigma_C,\delta_C])$ is 
associative;

\item $(A,B,C)$ is associative,
$\sigma_C$ is $\sigma_B$-twisted and
$\delta_C$ is a $\delta_B$-twisted
$\sigma_B$-derivation.
\end{enumerate}
\end{cor}

\begin{proof} 
The implication 
(ii)$\Rightarrow$(i)
follows from 
Theorem~\ref{thm:MimpliesMx}.
Now we show that
(i)$\Rightarrow$(ii).
Suppose that
$(A[x],B[x;\sigma_B,\delta_B],
C[x;\sigma_C,\delta_C])$ is 
associative.
Then, trivially, $(A,B,C)$ is 
associative.
Take $a \in A$, 
$b \in B$ and $c \in C$.
From $(ax)(bc)=((ax)b)c$ 
we get that
$a \sigma_C(bc) = a \sigma_B(b)\sigma_C(c)$
and $a \delta_C(bc)= 
a\sigma_B(b)\delta_C(c) + a\delta_B(b)c$.
Thus, both $\sigma_C(bc)-\sigma_B(b)\sigma_C(c)$
and 
$\delta_C(bc) - \sigma_B(b)\delta_C(c)
- \delta_B(b)c$ belong to $\Ann_C(A) = \{ 0 \}$.
Therefore, $\sigma_C$ is $\sigma_B$-twisted and
$\delta_C$ is a $\delta_B$-twisted
$\sigma_B$-derivation.
\end{proof}

\begin{ex}\label{ex:associative}
Suppose that $F$ is a field and $\sigma_F$ 
is a field endomorphism of $F$.
Let $V$ be an $F$-vector space with a 
fixed basis $(v_i)_{i \in I}$.
Let $\sigma_F$ act on $V$ in the following way.
Suppose that $v = \sum_{i \in I} f_i v_i$,
for some $f_i \in F$ with $f_i = 0$
for all but finitely many $i \in I$.
Put $\sigma_F(v) \colonequals \sum_{i \in I} 
\sigma_F(f_i) v_i$.
Let $\alpha$ be an $F$-vector space endomorphism of $V$.
Put $\sigma_V \colonequals \sigma_F \circ \alpha$.
Then $\sigma_V$ is a $\sigma_F$-twisted
endomorphism of $V$. Let $\delta_F\colon F \to F$ 
be any $\delta_F$-twisted $\sigma_F$-derivation 
on $F$, for instance $\delta_F \colonequals \id_F - \sigma_F$.
Define $\delta_V \colon V \to V$ by
$\delta_V(v) \colonequals \sum_{i \in I} \delta_F(f_i)v_i$,
for $v \in V$.
Then $\delta_V$ is a $\delta_F$-twisted 
$\sigma_F$-derivation on $V$.
By Corollary~\ref{cor:assocEquiv},
the triple $(F[x],F[x;\sigma_F,\delta_F], 
V[x;\sigma_V,\delta_V])$ is associative.
\end{ex}

\section{A Hilbert's basis theorem}\label{sec:hilbertsbasistheorem}

Throughout this section, 
$B$ is an abelian group 
with operators in a nonempty set $A$,
$+$ denotes the group operation in $B$,
and $0$ denotes the identity element of~$B$.
We also assume that $A$ has a zero element
(see Definition~\ref{defi:zeroelement}) and that 
$B[x;\sigma_B,\delta_B]$ is an Ore group
extension. As is customary,
given $z \in \Z$ and $b \in B$ we write
\[
zb \colonequals \left\{
\begin{array}{cccc}
b + \cdots + b \quad &\mbox{($z$ terms)}
& \mbox{if} & z>0, \\
0  && \mbox{if} & z=0, \\
(-b) + \cdots + (-b) \quad 
& \mbox{($-z$ terms)}
& \mbox{if} & z<0.
\end{array}
\right.
\]
Let $S$ be a subset of $B$. 
We let $\Z S$
denote the set of finite sums of
elements of the form $zs$ for 
$z \in \Z$ and $s \in S$.
Let $\{ S_i \}_{i \in I}$ be a family
of subsets of $B$. We let 
$\sum_{i \in I} S_i$ denote the 
set of finite sums of elements 
from $\cup_{i \in I} S_i$.
From Proposition~\ref{prop:generated} and
Proposition~\ref{prop:weak-s-unital-action},
we immediately get the following:

\begin{prop}\label{prop:abeliangenerated}
Let $S$ be a nonempty subset of $B$.
Then $\langle S \rangle =
\sum_{n \in \N} \Z (A^n S)$ and 
$[S] = \sum_{n \in \N_+} \Z (A^n S)$.
The action of $A$ on $B$
is weakly $s$-unital if and only if
for every nonempty subset $S$ of $B$
the equality $\langle S \rangle =
\sum_{n \in \N_+} \Z (A^n S)$ holds.
\end{prop}

\begin{defi}[Projection map]
Let $n \in \mathbb{N}$.
We define the \emph{projection map}
$\beta_n\colon B[x] \to B[x]$ by 
$\beta_n\left(  
\sum_{i\in\mathbb{N}} b_i x^i\right)  \colonequals b_n x^n$,
for $\sum_{i\in\mathbb{N}} b_i x^i \in B[x]$.
\end{defi}

\begin{lem}\label{lem:horrible}
Suppose that $B[x;\sigma_B,\delta_B]$ is an
Ore group extension, $b \in B$ and 
$i,j,k \in \mathbb{N}$. 
Then the following assertions hold:
\begin{enumerate}[{\rm (i)}]
\item
$\beta_{i+j} ( A^k ( (Ax^i)(b x^j) ) ) 
= ( A^{k+1} \sigma_B^i(b) ) x^{i+j}$;

\item if 
$\sigma_B^i(b) \in [ \sigma_B^i(b) ]$,
then $\sigma_B^i(b)  x^{i+j} \in 
\beta_{i+j} \left( 
\langle (Ax^i)(b x^j) \rangle \right)$.  
\end{enumerate}
\end{lem}

\begin{proof}
(i): It is clear that
\[
\beta_{i+j} ( A^k ( (Ax^i)(b x^j) ) )
\subseteq 
\beta_{i+j} ( A^k( (A 
\sigma_B^i(b))x^{i+j}))
\subseteq (A^{k+1} \sigma_B^i(b)) 
x^{i+j}.
\]
Now we show the reversed inclusion.
Take
$a_1,\ldots,a_{k+1} \in A$.
Then
\[
\begin{array}{rcl}
a_1 ( a_2 ( \cdots ( a_{k+1} \sigma_B^i(b)
) \cdots )) x^{i+j} &=& 
\beta_{i+j} \left(
a_1( a_2 ( \cdots ( a_k(
(a_{k+1} x^i) b x^j ) ) \cdots ) ) \right) 
\\[5pt]
&\in& \beta_{i+j} \left(
A^k( (Ax^i)(b x^j) ) \right). 
\end{array}
\]
(ii): Suppose  that
$\sigma_B^i(b) \in [ \sigma_B^i(b) ]
= \sum_{k \in \N}
\Z ( A^{k+1} \sigma_B^i(b) )$.
From (i) it follows that
\[\begin{array}{rcl}
\sigma_B^i(b) x^{i+j} &\in& 
\displaystyle
\sum_{k \in \N}
\Z \left( A^{k+1} \sigma_B^i(b) 
\right) x^{i+j} 
=
\displaystyle 
\beta_{i+j} \left(
\sum_{k \in \N} \Z \left( A^k \left(
(A x^i)(b x^j) \right) \right) \right) \\
 &=& \beta_{i+j} \left(
\langle (Ax^i)(b x^j) \rangle \right). 
\end{array}
\]
\end{proof}

\begin{thm}\label{thm:hilbertAB}
Suppose that $B[x;\sigma_B,\delta_B]$ is an
Ore group extension of a stably Noetherian 
group $B$ on which the action is weakly $s$-unital.
Let $\sigma_B$ be an $A$-stable 
surjection.
Then $B[x;\sigma_B,\delta_B]$ 
is stably Noetherian, seen as a group with 
operators in $A[x]$.
\end{thm}

\begin{proof}
Put $S \colonequals  A[x]$ and
$T \colonequals  B[x;\sigma_B,\delta_B]$.
Let $P$ be a nonzero $S$-stable subgroup 
of $T$. Set 
\[
Q \colonequals  \{ b \in B \mid \exists 
d \geq 0 \
\exists
b_{d-1},b_{d-2},\ldots,b_0 \in B \ \mbox{with} \
\sigma_B^d(b) x^d +
b_{d-1}x^{d-1} + \cdots + b_0
\in P \}.
\]
We now show that $Q$ is an $A$-stable
subgroup of $B$.
We first show that 
$AQ \subseteq Q$.
Take $a \in A$ and $b \in Q$.
There are $d \geq 0$ and
$b_{d-1},b_{d-2},\ldots,b_0 \in B$ such that
\[ p \colonequals  \sigma_B^d(b) x^d +
b_{d-1}x^{d-1} + \cdots + b_0
\in P.\]
Since $\sigma_B$ is $A$-stable,
there is $a' \in A$ with 
$\sigma_B^d(ab) = a' \sigma_B^d(b)$.
Thus
\[
\sigma_B^d(a b) x^d +
a' b_{d-1}x^{d-1} + 
\cdots + a' b_0 \\[5pt]
= 
a' \sigma_B^d(b) x^d +
a' b_{d-1}x^{d-1} + 
\cdots + a' b_0
\ = \ a' p \in P,
\]
so that $ab \in Q$. 
Now we show that $Q + Q \subseteq Q$.
To this end, take $b' \in Q$, $e \geq 0$
and
$b_{e-1}',b_{e-2}',\ldots,b_0' \in B$ with
$
p' \colonequals  \sigma_B^e(b') x^e +
b_{e-1}' x^{e-1} + \cdots + 
b_0' \in P. 
$
Suppose that $d \geq e$. 
Since the action on $B$ is weakly $s$-unital, it follows
from Lemma~\ref{lem:horrible} 
(with $i = d-e$ and $j = e$)
that there is $p'' \in 
\sum_{k \in \N} \Z (A^k((Ax^{d-e}) p')) 
\subseteq P$ with leading term
$\sigma_B^d(b') x^d$.
Since $p + p'' \in P$ and the leading term
of $p + p''$ is $\sigma_B^d(b) +
\sigma_B^d(b') = \sigma_B^d(b+b')$
it follows that $b + b' \in Q$.
The case $d < e$ is treated in a 
similar manner. 

Since $B$ is stably Noetherian, it follows 
from Proposition~\ref{prop:subquotient}
that $Q$ is stably Noetherian.
Thus, 
by Proposition~\ref{thm:threeequivalent}, there are $k \in \N_+$ and
$b_1,\ldots,b_k \in Q$ with
$Q = \langle b_1,\ldots,b_k \rangle$.
For each $i \in \{1,\ldots,k \}$ 
there is $p_i \in P$ with
leading term $\sigma^{n_i}(b_i) x^{n_i}$.
Let $n \colonequals {\rm max} \{ n_i \}_{i=1}^k$.
Fix $i \in \{ 1,\ldots,k \}$.
Since the action on $B$ is weakly $s$-unital, it follows
from Lemma~\ref{lem:horrible} 
(with $i = n-n_i$ and $j = n_i$)
that there is 
$p_i' \in \sum_{k \in \N} 
\Z (A^k((Ax^{n-n_i})p_i)) 
\subseteq P$ with leading term
$\sigma_B^n(b_i) x^n$.
Put 
$C \colonequals  \sum_{i=0}^{n-1} B x^i$.
By Proposition~\ref{prop:directproduct},
$C$ is stably Noetherian.
Therefore,
by Proposition~\ref{prop:subquotient},
$C \cap P$ is stably Noetherian, which in turn,
by Proposition~\ref{thm:threeequivalent},
implies that
$C \cap P$ is finitely generated by
some $a_1,\ldots,a_t \in C \cap P$.

Put $P_0 \colonequals
\langle p_1', \ldots , p_k' , 
a_1, \ldots , a_t \rangle$.
We claim that $P_0 = P$. 
Assuming that the claim
holds, $P$ is finitely generated, 
and hence, by Proposition
\ref{thm:threeequivalent},
$P$ is stably Noetherian.
Now we show the claim.
Trivially, $P_0 \subseteq P$.
Now we show the reversed inclusion.
To this end, take $p \in P$ and
let $D$ denote the degree of $p$.

Case 1: $D < n$. Then $p \in C \cap P
= \langle a_1,\ldots,a_t \rangle 
\subseteq P_0$.

Case 2: $D \geq n$. Suppose that all
elements of $P$ of degree less than $D$
belong to $P_0$. 
Since $\sigma_B$ is surjective,
there is $b \in B$ such that
the leading
term of $p$ is $\sigma_B^D(b) x^D$.
From $p \in P$
it follows that $b \in Q$. 
Since $Q = \langle  
b_1,\ldots,b_k \rangle$ and the action on
$B$ is weakly $s$-unital, it follows
from Proposition~\ref{prop:abeliangenerated} that
$b \in \sum_{i=1}^k \sum_{j \in \N} 
\Z (A^{j+1} b_i)$. Thus,
since $\sigma_B$ is $A$-stable, we get
\[
\sigma_B^D(b) \in 
\sigma_B^D \left(  
\sum_{i=1}^k \sum_{j \in \N} 
\Z (A^{j+1} b_i) \right) \subseteq 
\sum_{i=1}^k \sum_{j \in \N} 
\Z (A^{j+1} \sigma_B^D(b_i)).
\]
By $s$-unitality of the action on $B$, for each $i \in \{1,\ldots,k\}$ we have
$\sigma_B^D(b_i) \in 
[ \sigma_B^D( b_i ) ]$.
Hence,
by Lemma~\ref{lem:horrible}(ii), 
it follows that
\[
\begin{array}{rcl}
\sigma_B^D(b) x^D 
&\in& \sum_{i=1}^k \sum_{j \in \N} 
\Z (A^{j+1} \sigma_B^D(b_i)) x^D
\subseteq 
\beta_D \left(
\sum_{i=1}^k \Z (A^j( (A x^{D-n}) 
( \sigma_B^n(b_i) x^n )))
\right)
\\[5pt]
&=& 
\beta_D \left(
\sum_{i=1}^k \Z (A^j( (A x^{D-n}) p_i' ))
\right) 
\ \subseteq \ 
\beta_D( P_0 ).
\end{array}
\]
Thus, there is $\overline{p} \in P_0$
with leading term
$\sigma_B^D(b) x^D$.
Hence, $p - \overline{p} \in P$
and the leading term of 
$p - \overline{p}$ has degree less than $D$.
By the induction hypothesis it follows that
$p - \overline{p} \in P_0$. Thus, 
$p = \overline{p} + 
p - \overline{p} \in P_0 + P_0 
\subseteq P_0$.
\end{proof}

Note that under the assumptions of the 
previous theorem, the map 
$\sigma_B$ is necessarily 
\emph{bijective}, as 
established by Proposition~\ref{prop:surinj}.

\begin{thm}\label{thm:taulinear}
Suppose that $B[x;\sigma_B,\delta_B]$ is an
Ore group extension of a stably Noetherian 
group $B$ on which the action is $s$-unital.
Let $\tau$ be a map $A \to A$ and suppose that 
$\sigma_B$ is a $\tau$-twisted surjection.
Then $B[x;\sigma_B,\delta_B]$ 
is stably Noetherian, seen as a group 
with operators in $A[x]$.
\end{thm}

\begin{proof}
This follows from Proposition
\ref{prop:tautwisted} and
Theorem~\ref{thm:hilbertAB}. 
\end{proof}

The next result shows that the 
weakly $s$-unital condition in Theorem~\ref{thm:hilbertAB} cannot,
in general, be removed.

\begin{thm}\label{thm:polynomialnoether}
Consider $B[x]$ as an abelian group 
with operators in $A[x]$.
Then $B[x]$ is stably Noetherian if and 
only if $B$ is stably Noetherian and the 
action on $B$ is weakly $s$-unital.
\end{thm}

\begin{proof}
The ``if'' statement follows
immediately from Theorem~\ref{thm:hilbertAB}.
Now we show the ``only if'' statement.
We use the argument from 
\cite[p. 2201]{Varadarajan1982}.
Suppose that $B[x]$ is stably Noetherian.
First we show that $B$ is stably Noetherian.
Seeking a contradiction, suppose that
$B_1 \subsetneq B_2 \subsetneq \cdots$
is an infinite ascending chain of stable
subgroups of $B$. Then, clearly,
$B_1[x] \subsetneq B_2[x] \subsetneq 
\cdots$ is an infinite ascending chain
of stable subgroups of $B[x]$,
which is a contradiction.
Now we show that the action of $A$
on $B$ is weakly $s$-unital.
Seeking a contradiction,
suppose that there is $c \in B$
such that $c \notin [c]$.
Since $[c]$ is a stable subgroup of $B$
we can consider the quotient
group $D \colonequals B/[c]$. 
Then $D$ is nonzero, because the class 
$c + [c]$ in $B/[c]$ is nonzero. 
Let $E \colonequals \langle c + [c] \rangle$ 
be the $A$-stable subgroup in
$D$ generated by $c + [c]$.
Consider $E[x]$ as a group with operators in 
$A[x]$. Since $Ac \in [c]$, it follows 
that $A E = \{0\}$. Hence, 
$A[x] E[x] = \{0\}$.
Thus, the $A[x]$-stable subgroups of $E[x]$
are the same as the additive subgroups
of $E[x]$. But $E[x]$ is an infinite
direct sum of the nonzero subgroups
$E x^n$, for $n \in \N$, so $E[x]$
is not stably Noetherian.
By Proposition~\ref{prop:subquotient}, 
$D[x]$ is not stably
Noetherian. Therefore,
by Proposition~\ref{prop:subquotient}
again, $B[x]$ is not stably Noetherian,
which is a contradiction.
\end{proof}

\begin{rem}
Theorem~\ref{thm:polynomialnoether} extends
Varadarajan's theorem, see \cite[Thm.~A]{Varadarajan1982}, 
generalizing the result 
from polynomial modules to the context of abelian 
groups with operators.
\end{rem}

\begin{ex}
Let $I \colonequals \{ 1,2,\ldots \}$ be a nonempty countable set.
For each $n \in I$, let $p_n$ denote the $n$th odd
prime number, so that $p_1 = 3$, 
$p_2 = 5$, $p_3 = 7$, and so on.
For each $n \in I$, let $C_n$ denote the 
cyclic group of order $p_n$.
Suppose that $B$ denotes the abelian group of all 
sequences $(c_n)_{n \in I} \in \prod_{n \in I} C_n$ 
of $c_n \in C_n$ with $c_n = e$ for all but 
finitely many $n \in I$.
Let $A = \{ a\}$ be a singleton set. Define the 
action of $A$ on $B$ by $a(c_n)_{n \in I} =
(c_n^{-1})_{n \in I}$. 
By Example~\ref{ex:inverseaction}
and Theorem~\ref{thm:polynomialnoether},
$B[x]$ is stably Noetherian as an abelian 
group with operators in $A[x]$ if and only 
if $I$ is finite.
\end{ex}

\section{Ore module extensions}\label{sec:oremodules}

Throughout this section, $R$ is a \emph{ring}.
By this we mean that $R$ is an additive group
equipped with a map $R \times R \ni (r,s) \mapsto
r s \in R$. We will refer to this map 
as the \emph{ring multiplication}.
From now on, let $M$ be a left 
\emph{$R$-module}. By this we mean that $M$
is an additive group equipped with a map
$R \times M \ni (r,m) \mapsto r m \in M$.
We will refer to this map as the 
\emph{module multiplication}.
If $N$ is an additive subgroup of $M$
and the image of the restriction of the module multiplication to $R \times N$
is contained in $N$, then $N$ is called a \emph{submodule} of $M$.
For $m \in M$ and $n \in \N_+$, we let $R^n m$
denote the set of elements 
$r_1(r_2 ( \cdots (r_{n-1} (r_n m)) \cdots ))$
for $n \in \N_+$ and $r_1,\ldots,r_n \in R$.

\begin{defi}
With the above notation, the left 
$R$-module $M$ is said to be: 
\begin{itemize}

\item \emph{associative} if
$(r s) m = r (s  m)$
for $r,s \in R$ and $m \in M$;

\item \emph{left distributive} if 
$r (m + n) = r m + r n$ for 
$r \in R$ and $m,n \in M$;

\item \emph{right distributive} if
$(r + s) m = r m + s m$
for $r,s \in R$ and $m \in M$;

\item \emph{weakly $s$-unital} if for each $m \in M$,
the relation $m \in \sum_{n \in \N_+} \Z (R^n m)$
holds;

\item \emph{$s$-unital} if for every $m \in M$
there is $r \in R$ such that $m = rm$;

\item \emph{Noetherian} if any chain $N_0 \subseteq N_1 \subseteq \cdots \subseteq N_i \subseteq \cdots $ of $R$-submodules of $M$ eventually stabilizes, i.e. if there is some $k\in \N$ such that 
$N_i=N_k$ for every $i\geq k$.

\end{itemize}
The ring $R$ is said to be \emph{associative} (resp. \emph{left 
distributive}, \emph{right distributive},
\emph{(weakly) left $s$-unital}) if $R$
is associative (resp. left distributive, right distributive, (weakly) $s$-unital)
as a left module over itself.
\end{defi}

\begin{rem}\label{rem:Moduleproperties}
Let $M$ be a left $R$-module. Consider $R$ and $M$ as
sets equipped with an action of $R$, induced by
the ring and module multiplication, respectively.

\begin{enumerate}[{\rm (i)}]

\item The $R$-module $M$ is left distributive
if and only if $M$ is a group with operators in $R$, 
in the sense defined in 
Section~\ref{sec:groupswithoperators}.

\item 
Suppose that $M$ is left distributive.
The $R$-module $M$ is associative if
and only if the triple $(R,R,M)$ is 
associative in the sense of Definition~\ref{def:assoc}.

\item
Suppose that $R$ is left distributive. The ring $R$ is associative if and only if the triple $(R,R,R)$ is associative in the sense of Definition~\ref{def:assoc}.

\item Suppose that the $R$-module $M$ is 
left distributive. 
Then $M$ is weakly $s$-unital as an $R$-module 
if and only if the action of
$R$ on $M$ is weakly $s$-unital in the sense of 
Definition~\ref{def:sunital}.

\item Suppose that the $R$-module $M$ is associative
and left distributive. 
Then $M$ is $s$-unital as an $R$-module 
if and only if the action of
$R$ on $M$ is $s$-unital in the sense of 
Definition~\ref{def:sunital}.

\item Suppose that $R$ is associative as a ring. 

\begin{enumerate}[{\rm (a)}]

\item The ring $R$ is left (right) 
distributive as a left module over itself
if and only if $R$ is a 
left (right) \emph{near-ring} in the sense of 
\cite[Def.~1.1]{Pilz1983}.  

\item The $R$-module $M$ 
is right distributive and associative 
if and only if $M$ is a \emph{near module over $R$},
in the sense of \cite[Def.~1]{Rama2021}, or, equivalently, 
an \emph{$R$-group}, in the sense of 
\cite[Def.~1.17]{Pilz1983}.

\item  The $R$-module $M$ 
is left distributive and associative 
if and only if $M$ is a 
\emph{modified near module over $R$},
in the sense of \cite[Def.~2]{Rama2021}.

\end{enumerate}
\end{enumerate}
\end{rem}

From now on, $\sigma_R$ and $\delta_R$
are maps $R \to R$,
and
$\sigma_M$ and $\delta_M$
are additive maps $M \to M$.

\begin{defi}[Ore ring extension]
Suppose that the ring $R$ is left distributive.
By the \emph{Ore ring extension} 
$R[x ; \sigma_R , \delta_R]$ we mean the 
polynomial group $R[x]$ with a left action on itself,
the action being defined as in 
Definition~\ref{def:oreextension}, with $A = B = R$.
\end{defi}

\begin{defi}[Ore module extension]\label{def:oremodule}
Suppose that both the ring $R$ and the left $R$-module 
$M$ are left distributive.
By the \emph{Ore module extension} 
$M[x;\sigma_M,\delta_M]$ we mean the polynomial group
$M[x]$ equipped with the ring $R[x ; \sigma_R, \delta_R]$
as a set of operators, the action being defined as in 
Definition~\ref{def:oreextension} with $A = R$ and
$B = M$.
\end{defi}

\begin{prop}\label{prop:Mleftdistributive}
Let $R$ be a left distributive ring, and let $M$ be
a left distributive $R$-module. Suppose that 
$M$ is associative as a left $R$-module.
If $\delta_M$ is a $\sigma_R$-derivation on $M$
and $\sigma_M$ is $\sigma_R$-twisted,
then the Ore module extension
$M[x;\sigma_M,\delta_M]$ is left 
distributive and associative as a 
left $R[x;\sigma_R,\delta_R]$-module.
\end{prop}

\begin{proof}
This follows from Theorem~\ref{thm:MimpliesMx}
and Remark~\ref{rem:Moduleproperties}(ii).
\end{proof}

\begin{prop}
Let $R$ be a left distributive ring and $M$ be a modified near module over $R$.
If $\sigma_R$ is a ring endomorphism of $R$, 
$\delta_R$ is a $\sigma_R$-derivation on $R$,
$\delta_M$ is a $\sigma_R$-derivation on $M$
and $\sigma_M$ is $\sigma_R$-twisted,
then $M[x;\sigma_M,\delta_M]$ is a modified near 
left module over $R[x;\sigma_R,\delta_R]$.
\end{prop}

\begin{proof}
By Theorem~\ref{thm:MimpliesMx} and
Remark~\ref{rem:Moduleproperties}(iii), the ring
$R[x ; \sigma_R , \delta_R]$ is associative. 
The claim now follows from 
Remark~\ref{rem:Moduleproperties}(vi)(c) and
Proposition~\ref{prop:Mleftdistributive}.
\end{proof}

\begin{cor}
Let $R$ be a left distributive ring and $M$ be
a left distributive $R$-module. 
Suppose that $\Ann_M(R)=\{ 0 \}$.
Then the following assertions are equivalent:

\begin{enumerate}[{\rm (i)}]

\item the left $R[x;\sigma_R,\delta_R]$-module
$M[x;\sigma_M,\delta_M]$ is associative;

\item $M$ is associative, 
$\delta_M$ is a $\sigma_R$-derivation,
and $\sigma_M$ is $\sigma_R$-twisted.

\end{enumerate}
\end{cor}

\begin{proof}
This follows from Corollary~\ref{cor:assocEquiv}
and Remark~\ref{rem:Moduleproperties}(ii).
\end{proof} 

\begin{thm}\label{thm:HilbertM}
Suppose that $R$ is a left distributive ring,
and that
$M$ is a left distributive, $s$-unital, and 
Noetherian $R$-module.
Let $\sigma_M$ be a $\sigma_R$-twisted 
surjection. 
Then $M[x;\sigma_M,\delta_M]$ is
Noetherian as a left 
$R[x;\sigma_R,\delta_R]$-module. 
\end{thm}

\begin{proof}
This follows from Proposition~\ref{prop:tautwisted}
and Theorem~\ref{thm:hilbertAB}. 
\end{proof}

\begin{cor}\label{cor:HilbertM}
Let $R$ be a left distributive ring and $M$ be an $s$-unital, and Noetherian $R$-module.
Suppose that $\sigma_R$ is a ring 
automorphism of $R$
and $\alpha$ is an $R$-module automorphism 
of $M$
with $\sigma_M = \sigma_R \circ \alpha$.
Then $M[x;\sigma_M,\delta_M]$ is
Noetherian as a left 
$R[x;\sigma_R,\delta_R]$-module.
\end{cor}

\begin{proof}
Since, clearly, $\sigma_M$ is 
$\sigma_R$-twisted,
the claim follows from Theorem~\ref{thm:HilbertM}.
\end{proof}

\begin{ex}
Suppose that $D$ is a Dedekind domain. 
Let $I$ be an
ideal of $D$ and let $\sigma_D$ be a ring 
automorphism of $D$. 
If we put $\sigma_I \colonequals \sigma_D|_I$, then
$\sigma_I$ is $\sigma_D$-twisted.
Let $\delta_I\colon I \to I$ be any 
additive map. Since $D$ is, in particular,
Noetherian, it follows that $I$ is Noetherian
as a left $D$-module. By Corollary~\ref{cor:HilbertM}
it follows that
$I[x;\sigma_I,\delta_I]$ is Noetherian
as a left $D[x ; \sigma_D,\delta_D]$-module.
\end{ex}

\begin{ex}
Let $F$, $V$, $\sigma_F$, $\alpha$, 
$\sigma_V$, $\delta_F$, and $\delta_V$ be defined 
as in Example \ref{ex:associative}.
Suppose that $V$ is finite-dimensional and that $\sigma_F$ and $\alpha$ are surjective (and hence also injective).
By Corollary~\ref{cor:HilbertM},
$V[x;\sigma_V,\delta_V]$ is Noetherian as a left
$F[x ; \sigma_F , \delta_F]$-module.
\end{ex}

\begin{thm}\label{thm:leftdistributive}
Suppose that $R$ is a left distributive ring and that the left $R$-module $M$ is left distributive.
Then the left $R[x]$-module $M[x]$ is Noetherian 
if and only if $M$ is Noetherian and weakly $s$-unital.
\end{thm}

\begin{proof}
This follows from Theorem~\ref{thm:polynomialnoether}
and Theorem~\ref{thm:HilbertM}.
\end{proof}

\begin{cor}\label{cor:leftdistributive}
Suppose that $R$ is a left distributive
ring. Then $R[x]$ is left Noetherian if and only
if $R$ is left Noetherian and weakly left $s$-unital.
In particular, if $R$ is unital, then $R[x]$
is left Noetherian if and only if $R$ is 
left Noetherian.
\end{cor}

\begin{proof}
This follows from Theorem~\ref{thm:leftdistributive}.
\end{proof}

\begin{ex}\label{ex:leftdistributive}
There are many examples of classes of left 
distributive rings upon which we can apply 
our results above, for instance
Corollary~\ref{cor:leftdistributive}.

\begin{enumerate}[{\rm (i)}]

\item All unital and Noetherian rings which are 
alternative or Jordan.

\item Dickson's \emph{left near-fields}
(see \cite[p. 254]{Pilz1983}) are 
\emph{finite} rings, hence Noetherian.
Note that left near-fields are not always
right distributive.

\item Algebras generated by the so-called 
\emph{Cayley--Dickson doubling procedure}. 
These algebras are finite-dimensional vector spaces, 
hence Noetherian.
Note that these algebras are both left and right
distributive for all dimensions.

\item Algebras generated by the 
so-called \emph{Conway--Smith doubling procedure}
(see \cite{Lundstrom2012}). These algebras
are finite-dimensional vector spaces, hence Noetherian.
Note that these algebras are left distributive, 
but not right distributive from dimension 16 and onwards.

\end{enumerate}
\end{ex}

\begin{ex}
Suppose that $\mathbb{F}_2$ denotes 
the field with two elements 0 and 1.
Let $G = \{ R,P,S \}$ denote the 
commutative non-associative 
\emph{rock, paper, scissors}
magma defined by the relations
$R^2 = R$, $P^2 = P$, $S^2 = S$,
$RP = P$, $RS = R$, $PS = S$.
It is easily checked that the magma algebra
$\mathbb{F}_2[G]$ is a Boolean ring. 
In particular, $\mathbb{F}_2[G]$ is weakly left $s$-unital. Since $\mathbb{F}_2[G]$ 
is finite, it is Noetherian.
Therefore, by Corollary~\ref{cor:leftdistributive},  the 
ring $\mathbb{F}_2[G][x]$ is left Noetherian.
However, note that
$\mathbb{F}_2[G]$ is not left unital. 
Indeed, seeking a contradiction, let $x \colonequals aR+bP+cS$ denote a left
identity for some $a,b,c \in \mathbb{F}_2$.
In particular, $x R = R$ which implies
that $aR + bP + cR = R$ so that
$b=0$ and $a+c = 1$. Thus, $x = R$ or $x=S$.
Case~1: $x = R$. From the equality
$xS = S$ we get $RS = S$ which is a
contradiction.
Case~2: $x = S$. From the equality 
$x P = P$ we get $SP = P$ which is a
contradiction.
\end{ex}

\begin{ex}
Sometimes Ore ring extensions 
produce rings which are left
Noetherian but not right Noetherian. Indeed,
suppose that $R$ is a finite unital associative ring.
Consider $R \times R$ with its usual
addition but with a new product defined
by $(r,s)(t,u) \colonequals  (rt,ru)$. 
It is easy to check that $R \times R$
is an associative ring.
Furthermore,
$R \times R$ is left unital having all elements of the form $(1,s)$ 
as left units. However, since
$(0,1)(t,u) = (0,0)$, the ring $R \times R$
is not right unital. 
There are many derivations on $R \times R$.
Indeed, fix 
$(v,w) \in R \times R$.
The induced adjoint derivation
$\delta_{(v,w)}\colon 
R \times R \to R \times R$ is defined by
\[
\delta_{(v,w)}(r,s) =
(v,w)(r,s) - (r,s)(v,w) =
(vr-rv,vs-rw) = 
\left( [v,r] ,  
\left|
\begin{array}{cc}
   v  & w \\
   r  & s
\end{array}
\right|
\right)
\]
for all $(r,s) \in R \times R$.
Since $R \times R$ is finite it follows that $R \times R$ is both
left and right Noetherian. From Theorem~\ref{thm:HilbertM} it thus follows that
the ring $(R \times R)
[x; \id_{R \times R} , \delta_{(v,w)}]$
is left Noetherian.
For every $n \geq 0$ put
$I_n \colonequals \sum_{i=0}^n 
(\{ 0 \} \times A) x^i$. Then each 
$I_n$ is a right ideal in the ring
$(R \times R)
[x; \id_{R \times R} , \delta_{(v,w)}]$.
But $I_0 \subsetneq I_1 
\subsetneq \cdots$ showing that the ring 
$(R \times R)
[x; \id_{R \times R} , \delta_{(v,w)}]$
is not right Noetherian.
\end{ex}

\end{document}